\documentclass[a4paper]{amsart}
\usepackage{graphicx}
\usepackage[colorlinks, linkcolor= blue, citecolor= red]{hyperref}
\usepackage{pdfsync}
\usepackage{color}

\usepackage[T1]{fontenc}
\usepackage[utf8]{inputenc} 

\usepackage[all]{xy}

\usepackage{amsmath, amssymb, amsfonts, amscd, amsthm, mathrsfs}

\renewcommand{\phi}{\varphi}

\newcommand{\q}{\mathbb{Q}}
\newcommand{\f}{\mathbb{F}}

\renewcommand{\bar}{\overline}
\newcommand{\gt}{\mathcal{GT}}
\newcommand{\gtz}{\mathcal{GT}_{\!1}}
\newcommand{\Gal}{\operatorname{Gal}}
\newcommand{\gal}{\Gal(\bar\q / \q)}
\newcommand{\psl}{\operatorname{PSL}}
\newcommand{\ssl}{\operatorname{SL}}
\newcommand{\tr}{\operatorname{Tr}}
\newcommand{\T}{\mathscr{T}}
\newcommand{\E}{\mathbf{E}}
\newcommand{\PE}{P{\!}E}

\newtheoremstyle{pedro}{}{}{\itshape}{}{\sc}{~--}{ }{\thmname{#1}\thmnumber{ #2}\thmnote{ (#3)}}

\newtheoremstyle{pedrodef}{}{}{}{}{\sc}{~--}{ }{\thmname{#1}\thmnumber{ #2}\thmnote{ (#3)}}

\theoremstyle{pedro}
\newtheorem{lem}{Lemma}[section]

\newtheorem{thm}[lem]{Theorem}

\newtheorem{prop}[lem]{Proposition}

\newtheorem{coro}[lem]{Corollary}

\theoremstyle{remark}

\newtheorem{rmk}[lem]{Remark}

\theoremstyle{pedrodef}

\newtheorem{ex}[lem]{Example}

\title{The Grothendieck-Teichmüller group of~$\psl(2,q)$}

\author{Pierre Guillot}

\address{
Universit\'{e} de Strasbourg \& CNRS\\
Institut de Recherche Math\'{e}matique Avanc\'{e}e\\
7~Rue Ren\'{e} Descartes\\
67084 Strasbourg, France}

\email{guillot@math.unistra.fr}

\let\oldtocsection=\tocsection
\let\oldtocsubsection=\tocsubsection
\let\oldtocsubsubsection=\tocsubsubsection

\renewcommand{\tocsection}[2]{\hspace{0em}\oldtocsection{#1}{#2}}
\renewcommand{\tocsubsection}[2]{\hspace{2em}\oldtocsubsection{#1}{#2}}
\renewcommand{\tocsubsubsection}[2]{\hspace{2em}\oldtocsubsubsection{#1}{#2}}

\numberwithin{equation}{section}

\begin{document}

\maketitle

\begin{abstract}
We show that the Grothendieck-Teichmüller group of~$\psl(2, q)$, or more precisely the group~$\gtz(\psl(2, q))$ as previously defined by the author, is the product of an elementary abelian~$2$-group and several copies of the dihedral group of order~$8$. Moreover, when~$q$ is even, we show that it is trivial.

We explain how it follows that the moduli field of any ``dessin d'enfant'' whose monodromy group is~$\psl(2, q)$ has derived length~$\le 3$.

This paper can serve as an introduction to the general results on the Grothendieck-Teichmüller group of finite groups obtained by the author.
\end{abstract}

\section{Introduction \& Statement of results}

In~\cite{pedro}, we have introduced the {\em Grothendieck-Teichmüller group} of a finite group~$G$, denoted~$\gt(G)$. Motivation for the study of this group stems from the theory of {\em dessins d'enfants}. Recall that a dessin is essentially a bipartite graph embedded on a compact, oriented surface (without boundary), and that the absolute Galois group~$\gal$ acts on (isomorphism classes of) dessins. As explained in {\em loc.\ cit.}, there is an action of~$\gt(G)$ on those dessins whose {\em monodromy group} is~$G$, and the Galois action on the same objects factors {\em via} a map~$\gal \longrightarrow \gt(G)$.

Motivation for the study of {\em all} groups~$\gt(G)$, for all groups~$G$, is increased by the fact that the combined map 
\[ \gal \longrightarrow \gt := \lim_G \, \gt(G)  \]
is injective.

The group~$\gt(G)$ possesses a normal subgroup~$\gtz(G)$, which is such that the quotient~$\gt(G) / \gtz(G)$ is abelian. It follows that the commutator subgroup of~$\gal$ maps into~$\gtz(G)$, and injects into the inverse limit~$\gtz$ formed by these as~$G$ varies. There is little mystery left in~$\gt(G) / \gtz(G)$ (see~\cite{pedro} again), and the challenge is in the computation of~$\gtz(G)$.

In this paper we treat the case of~$G= \psl(2, q)$. We obtain the following result.

\begin{thm} \label{thm-main-intro}
The group~$\gtz(\psl(2, 2^s))$ is trivial for all~$s \ge 1$.

The group~$\gtz(\psl(2, q))$, when~$q$ is odd, is isomorphic to a product 
\[ C_2^{n_1} \times D_8^{n_2} \, .   \]
\end{thm}

Here~$D_8$ is the dihedral group of order~$8$. Note that this result was observed experimentally for small values of~$q$ in~\cite{pedro}. 




This theorem depends crucially on the work of MacBeath in~\cite{mac}, which classifies the triples~$(x, y, z)$ in~$\psl(2, q)$ in various ways. Indeed, we feel that the group~$\gtz(\psl(2, q))$ encapsulates part of this information neatly.  

Let us give an application to dessins d'enfants. The first part of the next theorem was implicit in~\cite{pedro}, and indeed it hardly deserves a proof once the statement is properly explained. However, it seems worth spelling it out for emphasis. 

\begin{thm} \label{thm-moduli}
Let~$G$ be a finite group. There exists a number field~$K$, Galois over~$\q$, such that~$\Gal(K/\q)$ is a subgroup of~$\gt(G)$, and containing the moduli field of any dessin whose monodromy group is~$G$.

For example, suppose that~$X$ is a dessin whose monodromy group is~$\psl(2, q)$. If~$q$ is even, then the moduli field of~$X$ is an abelian extension of~$\q$. If~$q$ is odd, then the Galois closure~$\tilde F$ of the moduli field~$F$ of~$X$ is such that~$\Gal(\tilde F/\q)$ has derived length~$\le 3$.
\end{thm}

A word of explanation. First, when~$\Gamma$ is a group we write~$\Gamma'$ for the derived (commutator) subgroup, and we say that~$\Gamma$ has derived length~$\le 3$ when~$\Gamma'''$ is trivial. Also, the {\em moduli field} of a dessin is the extension~$F$ of~$\q$ such that~$\Gal(\bar \q / F)$ is the stabilizer of the isomorphism class of~$X$ under the Galois action. Note that if we can write down explicit equations for~$X$ with coefficients in the number field~$L$, then certainly the moduli field~$F$ is a subfield of~$L$. While there are subtle counterexamples of dessins for which there are no equations over~$F$, it is still intuitively helpful to think of~$F$ as the smallest field over which the dessin is defined.

For example in~\cite{pedro-survey}, Example 4.6 and Example 4.13, we have examined a certain dessin~$X$ (a planar tree), whose monodromy group is the simple group of order~$168$, that is~$\psl(2, 7)$ (or $\psl(3, 2)$, as it is written in {\em loc.\ cit.}). We found explicit equations with coefficients in a field of the form~$\q(\alpha )$ with the minimal polynomial of~$\alpha $ having degree 4 (though not all details are provided); if~$L$ is the Galois closure of~$\q(\alpha )$, then~$\Gal(L/\q)$ is a subgroup of~$S_4$, which has derived length~$3$, confirming the prediction. However, there is even an easier way to see that the moduli field is very simple: there are only two dessins in the Galois orbit of~$X$, so the moduli field is in fact a quadratic extension of~$\q$.

It is an open problem to {\em explicitly} exhibit a dessin such that~$\Gal(\tilde F/\q)$ is non-abelian.

\bigskip

The examples treated in this paper are a less technical illustration of the ideas discussed in~\cite{pedro}, and may serve as an introduction to the latter. Note that, motivation and background aside, it is not necessary to be familiar with~\cite{pedro} in order to follow the arguments we present, leading to the computation of~$\gt_1(\psl(2, q))$.

\section{Definitions}

We take a definition of~$\gtz(G)$ which is only suitable when~$G$ is non-abelian and simple, such as~$G= \psl(2, q)$ ; see~\cite{pedro} for the more general definition. 

So let~$G$ be such a finite group, and let~$\T$ denote the set of triples~$(x, y, z) \in G^3$ such that~$xyz= 1$ and $\langle x, y, z \rangle = G$. Further, we let~$\T/G$ denote the set of orbits in~$\T$ under simultaneous conjugation by an element of~$G$. We write~$[x, y, z]$ for the class of~$(x, y, z)$. (In~\cite{pedro} we write~$\mathscr{P}$ instead of~$\T$, thinking of these elements as pairs~$(x, y)$.)

There is a free action of~$Out(G)$, the group of outer automorphisms of~$G$, on~$\T/G$. Moreover, there is also an action of~$S_3$, the symmetric group of degree~$3$. This is essentially a permutation of the coordinates, but to be more precise, one usually introduces the permutation~$\theta $ of~$\T/G$ defined by~$\theta \cdot [x, y, z]  = [ y, x, z^{x}]$, and the permutation~$\delta $ defined by~$\delta \cdot [x, y, z] = [z, y, x^{y}]$. These are both well-defined, and square to the identity operation of~$\T/G$. There is a homomorphism~$S_3 \to S(\T/G)$, where~$S(\T/G)$ is the symmetric group of the set~$\T/G$, mapping~$(12)$ to~$\theta $ and~$(13)$ to~$\delta $.

The two actions described commute, and together define an action of~$H :=Out(G) \times S_3$ on~$\T/G$. 

Let us write~$[x, y, z] \equiv [x', y', z']$ when~$x$ is a conjugate of~$x'$, while~$y$ is a conjugate of~$y'$, and~$z$ is a conjugate of~$z'$. This is an equivalence relation on~$\T/G$.

The group~$\gtz(G)$ is defined, in this context, to be the subgroup of the symmetric group~$S(\T/G)$ comprised by those permutations~$\phi$ which:

\begin{itemize}
\item commute with the action of~$H$; in other words, if~$h \in H$, $t \in \T/G$ then~$\phi(h\cdot t) = h \cdot \phi(t)$.
\item are compatible with~$\equiv$; that is, $t \equiv t'$ implies~$\phi(t) \equiv \phi(t')$, if~$t, t' \in \T/G$.
\end{itemize}

\noindent (Somewhat arbitrarily, we write~$h \cdot t$ for the action of~$h \in H$, and~$\phi(t)$ for the action of~$\phi \in \gtz(G)$, in order to set the elements of~$\gtz(G)$ apart.)

\section{Characteristic two}

We start by assuming that~$q$ is a power of~$2$, so that~$\psl(2, q) = \ssl(2,q)$.

Following MacBeath~\cite{mac}, we partition the set of triples~$(x, y, z)$ of elements of~$\ssl(2, q)$ satisfying~$xyz=1$ into the subsets~$\E(a, b, c)$, where~$a, b, c \in \f_q$, by requiring~$(x, y, z) \in \E(a, b, c)$ when~$\tr(x)= a$, $\tr(y)=b$, $\tr(z)=c$ (here~$\tr$ is the trace).

Since elements of~$\gtz(\ssl(2, q))$ are assumed to be compatible with the relation~$\equiv$, the following observation is trivially true.

\begin{lem}
Suppose~$(x, y, z) \in \E(a, b, c)$, with~$\langle x, y, z \rangle = \ssl(2, q)$, let~$\phi \in \gtz(\ssl(2, q))$, and suppose that~$x', y', z'$ satisfy 
\[ \phi([x, y, z]) = [x', y', z'] \, .   \]
Then~$(x', y', z') \in \E(a, b, c)$. \hfill $\square$
\end{lem}

Note that~$\ssl(2, q)$ acts on~$\E(a, b, c)$ by simultaneous conjugation. The crucial point is this:

\begin{prop}[after MacBeath] \label{prop-mac-mod-2}
When the set~$\E(a, b, c)$ contains a triple~$(x, y, z)$ such that~$\langle x, y, z \rangle = \ssl(2, q)$, it consists of just one conjugacy class.
\end{prop}

\begin{proof}
In~\cite{mac}, the triples~$(a, b, c)$ are divided into the ``singular'' ones and the ``non-singular'' ones ; also, the type of~$(x, y, z)$ is the type of~$(\tr(x), \tr(y), \tr(z))$ by definition. Theorem 2 asserts that when~$(x, y, z)$ is singular, the group~$\langle x, y, z \rangle$ is ``affine'', and in particular it is not all of~$\ssl(2, q)$. Our hypothesis guarantees thus that~$(a, b, c)$ is non-singular.

We may then apply (ii) of Theorem 3 in {\em loc.\ cit.}, giving the result.
\end{proof}

\begin{coro}
The group~$\gtz(\ssl(2, q))$ is trivial.
\end{coro}

\begin{proof}
Let~$\phi \in \gtz(\ssl(2, q))$. Any~$t \in \T/G$ is of the form~$t= [x, y, z]$ with~$(x, y, z) \in \E(a, b, c)$ for some~$a, b, c$, and~$\langle x, y, z \rangle = \ssl(2, q)$ by definition. The Lemma applies, showing that~$\phi(t) = [x', y', z']$ with~$(x', y', z') \in \E(a, b, c)$, while the Proposition proves that all triples in~$\E(a, b, c)$ are in fact conjugate. As a result~$\phi(t) = t$.
\end{proof}

\section{Odd characteristics}

Now we assume that~$q = p^s$ is a power of the odd prime~$p$, and we turn to the description of~$\gtz(G)$ where~$G= \psl(2, q)$. 

\subsection{Sets of triples}

As in the previous section, we define~$\E(a, b, c)$ to be the set of triples~$(x, y, z) \in \ssl(2, q)^3$ such that~$xyz= 1$ and with~$\tr(x)= a$, 
$\tr(y)=b$, $\tr(z)=c$. We also define~$E(a, b, c)$ to be the subset of~$\E(a, b, c)$, which may well be empty, of triples generating~$\ssl(2, q)$ (or equivalently, whose images generate~$G$). Finally, we write~$\PE(a, b, c)$ for the image of~$E(a, b, c)$ in~$G^3$.

\begin{lem} The notation behaves as follows. \begin{enumerate}

\item If~$\PE(a, b, c)$ and~$\PE(a', b', c')$ are not disjoint, then they are equal, and $(a', b', c') = (\pm a, \pm b, \pm c)$ for some choices of signs. 

\item We have 
\[ \PE(a, b, c) = \PE(-a, -b, c) = \PE(-a, b, -c) = \PE(a, -b, -c) \, .    \]
In other words, the set~$\PE(a, b, c)$ is not altered when an even number of signs are introduced.

\item When~$abc = 0$, all choices of signs give the same set~$\PE(\pm a, \pm b, \pm c)$.

\item When~$abc \ne 0$, the sets~$\PE(a, b, c)$ and~$\PE(a, b, -c)$ are disjoint.

\end{enumerate}

\end{lem}

\begin{proof}
(1) An element~$(g,h,k) \in \PE(a, b, c)$ is of the form~$(\bar x, \bar y, \bar z)$, where~$x, y, z \in \ssl(2, q)$ and the bar denotes the morphism to~$G$, where the traces of these elements are~$a, b, c$ respectively. If~$(g, h, k)$ also belongs to~$\PE(a', b', c')$, given that the possible lifts of~$g, h, k$ are~$\pm x, \pm y, \pm z$ respectively, we see that~$a'= \pm a$, $b' = \pm b$, $c' = \pm c$. The fact that~$\PE(a, b, c) = \PE(a', b', c')$ will follow from (2)-(3)-(4) (since these properties imply that~$\PE(a, b, c)$ and~$\PE(\pm a, \pm b, \pm c)$ are either equal or disjoint).

(2) If~$(x, y, z) \in E(a, b, c)$, then~$(-x, -y, z) \in E(-a, -b, c)$, and these two triples map to the same element in~$G^3$. This shows that an element of~$\PE(a, b, c)$ also belongs to~$\PE(-a, -b, c)$, and conversely. The other arguments are similar.

(3) If~$abc= 0$, then one of~$a, b, c$ is~$0$, say~$a=0$, so that~$a = -a$. We are thus free to change the sign of~$a$, and an even number of other signs, which gives the result.

(4) If~$x' = \pm x$, and~$\tr(x') = \tr(x) \ne 0$, then~$x'= x$. We see thus that, whenever two triples~$(x, y, z) \in E(a, b, c)$ and~$(x', y', z') \in E(a, b, -c)$ map to the same element of~$G^3$, we must have~$x' = x$ and $y' = y$, so that~$z'= z$ since~$xyz= 1 = x'y'z'$. This is a contradiction since the traces of~$z$ and~$z'$ are~$c \ne 0$ and~$-c$. As a result, $\PE(a, b, c)$ and $\PE(a, b, -c)$ are disjoint in this case.
\end{proof}

\begin{ex}
Trying the example of~$\psl(2, 5)$, one finds that~$\PE(0, 2, 3)$ is non-empty, showing that the case~$abc= 0$ does occur non-trivially. The set~$\PE(2, 2, 4)$ is also non-empty, as is~$\PE(2, 2, -4)$, so the case~$abc \ne 0$ occurs and states here the disjointness of non-empty sets. However, $\PE(1, 2, 4)$ is non-empty, but $\PE(1, 2, -4)$ {\em is} empty, an instance where (4) still holds, but in a degenerate way.
\end{ex}

We define finally 
\[ \T(a, b, c) = \bigcup_{\textnormal{signs}} \PE(\pm a, \pm b, \pm c) = \PE(a, b, c) \cup \PE(a, b, -c) \, .   \]
This is a subset of~$\T$, and~$\T(a, b, c)/G$ is a subset of~$\T/G$. As~$(a, b, c)$ varies, the subsets~$\T(a, b, c)/G$ are disjoint, and constitute an initial partition of~$\T/G$.

\begin{lem}
The subset~$\T(a, b, c)/G$ is stable under the action of~$\gtz(G)$.
\end{lem}

\begin{proof}
Suppose~$\phi \in \gtz(G)$, and~$\phi([g, h, k]) = [g', h', k']$, with $g, h, k, g', h', k' \in G$. Since~$\phi$ is compatible with~$\equiv$ by definition, we see that~$g'$ is conjugate to~$g$ within~$G$; writing~$g = \bar x$ for~$x \in \ssl(2, q)$, and similary~$g' = \bar{x'}$, we conclude that~$x'$ is a conjugate of~$\pm x$, so~$\tr(x') = \pm \tr(x)$. Similar considerations apply to~$h$ and~$h'$, and to~$k$ and~$k'$.

We conclude that if~$(g, h, k) \in \PE(a, b, c)$, then~$(g', h', k') \in \PE(\pm a, \pm b, \pm c)$, as we wanted.
\end{proof}

\begin{rmk} \label{rmk-Tabc-equiv}
Similar arguments show that~$\T(a, b, c)/G$ is a union of equivalence classes for~$\equiv$.
\end{rmk}

\subsection{Number of conjugacy classes of triples}

The action of~$G$ on~$\T$ by (simultaneous) conjugation restricts to an action on each set~$\PE(a, b, c)$, clearly. Moreover, let us introduce the automorphism~$\alpha $ of~$G$ induced by conjugation by 
\[ \left(\begin{array}{rr}
1 & 0 \\
0 & -1
\end{array}\right) \in \operatorname{GL}(2, q) \smallsetminus \ssl(2, q) \, .  \]
One verifies that~$\alpha $ is not inner (below we recall the description of~$Out(G)$). Moreover, since conjugate matrices have the same trace, we see that the action of~$\alpha $ on the triples in~$\T$ also preserves the sets~$\PE(a, b, c)$.

\begin{prop}[after MacBeath]
When~$\PE(a, b, c)$ is non-empty, it is made of precisely two conjugacy classes, which are exchanged by~$\alpha $.
\end{prop}

\begin{proof}
First we argue as in Proposition~\ref{prop-mac-mod-2}, relying on (i) of Theorem 3 in~\cite{mac}. The conclusion is that when~$E(a, b, c)$ is non-empty, that is when~$\E(a, b, c)$ contains a triple generating~$\ssl(2,q)$, then~$\E(a, b, c)$ consists of two conjugacy classes exactly.

If~$(x, y, z) \in E(a, b, c)$, then~$(\alpha (x), \alpha (y), \alpha (z))$ cannot be in the conjugacy class of~$(x, y, z)$, lest we should conclude that~$\alpha $ is inner (here we view~$\alpha $ as an automorphism of~$\ssl(2, q)$, rather than~$G$). However~$(\alpha (x), \alpha (y), \alpha (z)) \in E(a, b, c)$, showing that~$E(a, b, c)$ intersects both conjugacy classes in~$\E(a, b, c)$, and that~$E(a, b, c) = \E(a, b, c)$.

When~$\alpha $ is viewed as an automorphism of~$G$, it is still non-inner. So the same reasoning applies, showing that there are triples in~$\PE(a, b, c)$ which are not conjugate to one another, and more precisely that~$(g, h, k)$ and~$(\alpha (g), \alpha (h), \alpha (k))$ are never conjugate. The Proposition has been proved.
\end{proof}

The cardinality of~$\PE(a, b, c)/G$ is thus~$2$, when it is not~$0$; and~$\T(a, b, c)/G$ contains~$2$ or~$4$ elements (or~$0$). These sets are unions of orbits of~$\alpha $ (recall that~$Out(G)$ acts freely on~$\T/G$).

\subsection{The action of~$H$}

Recall that we write~$H= Out(G) \times S_3$. According to~\cite{wilson}, Theorem 3.2, when~$G= \psl(2, p^s)$ with~$p$ odd, we have~$Out(G) = \langle \alpha \rangle \times \Gal(\f_{p^s} / \f_p) \cong C_2 \times C_s$. Here~$\alpha $ is as above, and the Galois group acts on matrix entries in the obvious way. In particular, note that~$\alpha $ is central in~$H$.

Now suppose that~$(a, b, c)$ is a fixed triple, and let~$H_0$ denote the subgroup of~$H$ leaving the subset~$\T(a, b, c)/G$ stable, assuming the latter is non-empty. Note that~$\alpha \in H_0$.

\begin{lem} \label{lem-centralizer-D8}
The permutation group induced by~$H_0$ on the set~$\T(a, b, c)/G$ is isomorphic to either~$C_2$, or~$C_2^2$, or~$D_8$. The same can be said of the centralizer of this permutation group in the symmetric group~$S(\T(a, b, c)/G)$.
\end{lem}

\begin{proof}
If~$\T(a, b, c)/G$ has only~$2$ elements, there is nothing to prove, so we turn to the alternative, namely, we assume that this set has~$4$ elements. These are freely permuted by~$\alpha $, which has order~$2$, so they may be numbered~$1, 2, 3, 4$ in such a way that~$\alpha $ acts as~$(12)(34)$.

The centraliser of~$\alpha $ in~$S_4$ is isomorphic to~$D_8$, generated, say, by~$(12)$ and~$(13)(24)$. Since~$\alpha $ is central in~$H$, we have a map~$H_0 \longrightarrow D_8$, and the first part of the Lemma is about its image. The non-trivial subgroups of~$D_8$ are all of the form indicated, {\em except} for the presence of cyclic groups of order~$4$.

So we assume that 
\[  h= \alpha^i \sigma \pi \in \langle \alpha \rangle \times \Gal(\f_q/\f_p) \times S_3 = H  \]
belongs to~$H_0$ and acts as a~$4$-cycle on~$\T(a, b, c)/G$, and work towards a contradiction.

First, we may replace~$h$ by~$\alpha h$ if necessary, and assume that~$i= 0$, that is~$h= \sigma \pi$. The element~$\pi \in S_3$ has order~$1$, $2$ or~$3$; if it has order~$3$, we replace~$h$ by~$h^3 = \sigma^3 \pi^3 = \sigma^3$ and we are reduced to the case when~$\pi= 1$. So we assume that the order of~$\pi$ divides~$2$.

Elements of order~$4$ in~$D_8$, when squared, give the non-trivial central element, here~$(12)(34)$. Thus~$h^2 = \sigma^2$ acts as~$\alpha $ does. However, this is a contradiction, since~$\alpha $ and~$\sigma $ belong to~$Out(G)$, which acts {\em freely} on~$\T/G$, while~$\alpha = \sigma ^2$ does not hold.

This proves the first part. For the second part, since~$\alpha \in H_0$, we note that the centralizer in question must centralize~$(12)(34)$, so it is a subgroup of the~$D_8$ under consideration. The centralizer, in~$D_8$, of a subgroup which is not cyclic of order~$4$ is again not cyclic of order~$4$, as is readily checked.
\end{proof}

\subsection{The partition of~$\T/G$} We now let 
\[ X(a, b, c)  = \bigcup_{h \in H} h \cdot \T(a, b, c)/G \, .   \]
As~$a, b, c$ vary, the subsets~$X(a, b, c)$ provide a partition of~$\T/G$. Note that, given the description of~$H$ (and~$Out(G)$), we certainly have, for any~$h \in H$, 
\[ h \cdot \T(a, b, c)/G  = \T(a', b', c')/G \]
for some~$a', b', c'$.

\begin{lem}
Let~$\gtz(G)_{abc}$ be the permutation group on~$X(a, b, c)$, consisting of those permutations commuting with the action of~$H$, and compatible with the relation~$\equiv$. Then~$\gtz(G)$ is the direct product of the various groups~$\gtz(G)_{abc}$.
\end{lem}

\begin{proof}
This is a completely general fact: when~$\T/G$ is partitioned into subsets which are stable under the action of~$H$, and which are unions of equivalence classes for~$\equiv$, then~$\gtz(G)$ splits as a corresponding direct product, as one sees from the definition.
\end{proof}

Now suppose~$a, b, c$ are fixed, and resume the notation~$H_0$ from the previous section.

\begin{lem}
The permutation group~$\gtz(G)_{abc}$ is isomorphic to one of~$\{ 1 \}$, $C_2$, $C_2^2$, or~$D_8$.
\end{lem}

\begin{proof}
Since the action of~$\gtz(G)_{abc}$ commutes with that of~$H$, it is determined by its restriction to~$\T(a, b, c)/G$. In other words, the map~$\gtz(G)_{abc} \to S(\T(a, b, c)/G)$, which is well-defined since~$\T(a,b,c)/G$ is stable under~$\gtz(G)$, is injective.

The image~$\Gamma $ of that map is a permutation group which commutes with the action of~$H_0$, and so by Lemma~\ref{lem-centralizer-D8} it is a subgroup of either~$C_2$, $C_2^2$ or~$D_8$. Thus it remains to prove that~$\Gamma $ is not cyclic of order~$4$, which potentially could happen when the centralizer~$C$ of~$H_0$ is isomorphic to~$D_8$.

Indeed, suppose~$\Gamma $ contains a~$4$-cycle. We infer that~$\gtz(G)$ acts transitively on~$\T(a, b, c)/G$. It follows that the equivalence relation~$\equiv$, preserved by~$\gtz(G)$, is trivial, in the sense that it has just one class in this set: all the triples in~$\T(a,b,c)$ are ``coordinate-wise conjugate''. Thus the same can be said of~$\equiv$ on all the translates $h \cdot \T(a, b, c)/G$, easily. As a result, these translates are precisely the equivalence classes of~$\equiv$ on~$X(a, b, c)$ (see Remark~\ref{rmk-Tabc-equiv}).

However, let us now consider the action of the full centralizer~$C\cong D_8$, extended to all of~$X(a, b, c)$ by requiring commutation with the action of~$H$. Given the description of the classes of~$\equiv$, it is clear that~$C$ is compatible with this equivalence relation. We conclude that~$\gtz(G)_{abc}$ contains a copy of~$D_8$, and in particular it is not cyclic of order~$4$.
\end{proof}

The last two lemmas establish that, as announced:

\begin{thm}
When~$q$ is a power of an odd prime, there exist integers~$n_1, n_2$ such that 
\[ \gtz(\psl(2,q)) \cong C_2^{n_1} \times D_8^{n_2} \, .   \]
\end{thm}

In~\cite{pedro}, explicit examples have been computed (with the help of the GAP software). We found the following table. 
\[ \begin{array}{c|c|c}
q & n_1 & n_2 \\
\hline
5 & 0 & 0 \\
7 & 3 & 2 \\
9 & 12 & 1 \\
11 & 27 & 7 \\
13 & 54 & 17 \\
17 & 104 & 50 \\
19 & 133 & 74
\end{array}  \]
The first line is in accordance with the isomorphism~$\psl(2, 5) \cong \psl(2, 4)$.

\section{Application to dessins}

We will conclude the paper with a proof of Theorem~\ref{thm-main-intro}. Recall that~$\gal$ acts on the isomorphism classes of dessins, and that the action on those dessins with monodromy group~$G$ factors via a certain map 
\[ \lambda_G \colon \gal \longrightarrow \gt(G) \, .   \]
If~$K$ is that field such that~$\Gal(\bar\q / K) = \ker(\lambda_G)$, then~$K/\q$ is Galois and~$\Gal(K/\q)$ is identified with a subgroup of~$\gt(G)$. 

The moduli field of the dessin~$X$ is that field~$F$ such that~$\Gal(\bar\q / F)$ is the subgroup of elements stabilizing~$X$ (up to isomorphism). This subgroup contains~$\ker(\lambda_G)$ if the monodromy group of~$X$ is~$G$, so that~$F \subset K$. This proves the first part of the Theorem.

Now we specialize to~$G= \psl(2, q)$. If~$q$ is even, then~$\gtz(G) = 1$, so that~$\gt(G)$ is abelian (since the commutators belong to~$\gtz(G)$). In this case~$K/\q$ is an abelian extension of~$\q$, as is~$F/\q$ in the notation above. 

When~$q$ is odd, we can at least state that~$\gtz(G)$ is of derived length~$\le 2$. As a result, the derived length of~$\gt(G)$ is~$\le 3$. The same can be said of~$\Gal(K/\q)$ and of~$\Gal(\tilde F/\q)$, where~$\tilde F \subset K$ is the Galois closure of~$F$.

\bibliography{myrefs}

\newcommand{\noopsort}[1]{} \newcommand{\printfirst}[2]{#1}
  \newcommand{\singleletter}[1]{#1} \newcommand{\switchargs}[2]{#2#1}
  \def\cprime{$'$}
\providecommand{\bysame}{\leavevmode\hbox to3em{\hrulefill}\thinspace}
\providecommand{\MR}{\relax\ifhmode\unskip\space\fi MR }
\providecommand{\MRhref}[2]{%
  \href{http://www.ams.org/mathscinet-getitem?mr=#1}{#2}
}
\providecommand{\href}[2]{#2}
\begin{thebibliography}{Mac69}

\bibitem[Gui]{pedro}
Pierre Guillot, \emph{The {G}rothendieck-{T}eichm\"uller group of a finite
  group and {$G$}-dessins d'enfants}, to appear in the proceedings volume
  Symmetry in Graphs, Maps and Polytopes 2014, arXiv 1407.3112.

\bibitem[Gui14]{pedro-survey}
\bysame, \emph{An elementary approach to dessins d'enfants and the
  {G}rothendieck-{T}eichm\"uller group}, Enseign. Math. \textbf{60} (2014),
  no.~3-4, 293--375. \MR{3342648}

\bibitem[Mac69]{mac}
A.~M. Macbeath, \emph{Generators of the linear fractional groups}, Number
  {T}heory ({P}roc. {S}ympos. {P}ure {M}ath., {V}ol. {XII}, {H}ouston, {T}ex.,
  1967), Amer. Math. Soc., Providence, R.I., 1969, pp.~14--32. \MR{0262379}

\bibitem[Wil09]{wilson}
Robert~A. Wilson, \emph{The finite simple groups}, Graduate Texts in
  Mathematics, vol. 251, Springer-Verlag London, Ltd., London, 2009.
  \MR{2562037}

\end{thebibliography}
\bibliographystyle{amsalpha}

\end{document}